\documentclass[a4paper,11pt]{article}%
\usepackage{graphicx}
\usepackage[T1]{fontenc}
\usepackage[utf8]{inputenc}
\usepackage[english]{babel}
\usepackage{amsmath,amsthm,amssymb,amsfonts}
\usepackage{lmodern}
\usepackage{subcaption}
\usepackage{bm}

\usepackage{csquotes}

\usepackage{geometry}
 \usepackage{empheq} 
 \usepackage{xcolor}
 \definecolor{darkgreen}{rgb}{0,0.5,0}
 \definecolor{darkblue}{rgb}{0,0.08,0.45}

\usepackage{hyperref}
\hypersetup{
        colorlinks=true,
        linkcolor=darkblue,
        citecolor=darkblue,
        urlcolor=darkgreen
    }

\usepackage{algorithm}
\usepackage{algorithmic}
\usepackage{enumerate}

\usepackage{amsthm}
\usepackage{aliascnt}
\usepackage{cleveref}

\numberwithin{equation}{section}

\newtheorem{theorem}{Theorem}[section]

\newaliascnt{lemma}{theorem}
\newtheorem{lemma}[lemma]{Lemma}
\aliascntresetthe{lemma}

\newaliascnt{proposition}{theorem}
\newtheorem{proposition}[proposition]{Proposition}
\aliascntresetthe{proposition}

\newaliascnt{corollary}{theorem}
\newtheorem{corollary}[corollary]{Corollary}
\aliascntresetthe{corollary}

\theoremstyle{definition}

\newaliascnt{definition}{theorem}
\newtheorem{definition}[definition]{Definition}
\aliascntresetthe{definition}

\newaliascnt{example}{theorem}

\aliascntresetthe{example}

\newaliascnt{remark}{theorem}
\newtheorem{remark}[remark]{Remark}
\aliascntresetthe{remark}

\crefname{theorem}{theorem}{theorems}
\Crefname{theorem}{Theorem}{Theorems}

\crefname{proposition}{proposition}{propositions}
\Crefname{proposition}{Proposition}{Propositions}

\crefname{lemma}{lemma}{lemmas}
\Crefname{lemma}{Lemma}{Lemmas}

\crefname{corollary}{corollary}{corollaries}
\Crefname{corollary}{Corollary}{Corollaries}

\crefname{definition}{definition}{definitions}
\Crefname{definition}{Definition}{Definitions}

\crefname{example}{example}{examples}
\Crefname{example}{Example}{Examples}

\crefname{remark}{remark}{remarks}
\Crefname{remark}{Remark}{Remarks}

\definecolor{darkgreen}{rgb}{0,0.5,0}
\definecolor{darkblue}{rgb}{0,0.08,0.45}
\definecolor{someblue}{rgb}{0,0.08,0.9}
\definecolor{rust}{rgb}{0.6,0.1,0.1}

\usepackage{hyperref}
\hypersetup{
        colorlinks=true,
        linkcolor=blue,
        citecolor=rust,
        urlcolor=darkgreen
      }

\usepackage{tcolorbox,tikz}
\tcbuselibrary{skins}
\tcbuselibrary{breakable}

\newcommand{\one}{\bm{1}}
\newcommand{\zero}{\bm{0}}

\newcommand\R{\mathbb{R}}

\renewcommand{\a}{\alpha}

\renewcommand{\phi}{\varphi}

\newcommand{\bx}{\boldsymbol{x}}
\newcommand{\by}{\boldsymbol{y}}
\newcommand{\bz}{\boldsymbol{z}}
\newcommand{\bs}{\boldsymbol{s}}
\newcommand{\bb}{\boldsymbol{b}}
\newcommand{\bu}{\boldsymbol{u}}
\newcommand{\bv}{\boldsymbol{v}}
\newcommand{\bw}{\boldsymbol{w}}

\newcommand{\lr}[1]{\langle #1\rangle}
\newcommand{\lrb}[1]{\left( #1 \right)}
\newcommand{\lrset}[1]{\left\{ #1 \right\}}

\newcommand{\norm}[1]{\left\lVert#1\right\rVert}
\newcommand{\abs}[1]{\left\lvert#1\right\rvert}
\newcommand{\inner}[2]{\left\langle#1,#2\right\rangle}

\newcommand{\normi}[1]{\lVert#1\rVert_\infty}

\DeclareMathOperator{\argmin}{argmin}

\DeclareMathOperator{\range}{range}
\DeclareMathOperator{\diag}{diag}

\renewcommand{\leq}{\leqslant}
\renewcommand{\geq}{\geqslant}

\title{Entropic Mirror Descent for Linear Systems: Polyak's Stepsize and Implicit Bias}
\author{
  Yura Malitsky\thanks{University of Vienna, Austria. \href{mailto:yurii.malitskyi@univie.ac.at}{yurii.malitskyi@univie.ac.at}}
  \and Alexander Posch\thanks{University of Vienna, Austria. \href{mailto:alexander.posch@univie.ac.at}{alexander.posch@univie.ac.at}}}
\date{}

 \usepackage[numbers]{natbib}

\begin{document}
\maketitle
\begin{abstract}
This paper focuses on applying entropic mirror descent to solve linear systems, where the main challenge for the convergence analysis stems from the unboundedness of the domain. To overcome this without imposing restrictive assumptions, we introduce a variant of Polyak-type stepsizes. Along the way, we give an overview of bounds for $\ell_1$-norm implicit bias, obtain sublinear and linear convergence results, and generalize the convergence result to arbitrary convex $L$-smooth functions. We also propose an alternative method that avoids exponentiation, resembling gradient descent with Hadamard overparametrization but with provable convergence. 
\end{abstract}
\section{Introduction}
One focus of recent research in optimization has been to understand the
\emph{implicit bias} of optimization algorithms~--- the tendency of an algorithm
to converge to solutions with certain properties. In this paper, we continue
this line of research by studying entropic mirror descent applied to the linear
system $A \bx = \bb$. Specifically, given an $m \times n$ matrix $A$ and a
vector $\bb \in \R^m$, we analyze the convergence of entropic mirror descent,
defined by
\begin{equation}
  \label{eq:md}
  \bx_{k+1} = \bx_k \circ \exp(-\alpha_k \nabla f(\bx_k)),
\end{equation}
where $\alpha_k > 0$ is the stepsize, $f(\bx) = \frac{1}{2} \norm{A\bx - \bb}^2$, and $\circ$ denotes elementwise multiplication.

Our motivation for studying this scheme comes from its connection to gradient
descent applied to the nonconvex problem of minimizing $\tilde{f}(\bx) = \frac{1}{2} \norm{A(\bx\circ\bx) - \bb}^2\!$. For both approaches, the
implicit bias is fairly well understood~\cite{chou2023more,soltanolkotabi_implicit_2023}~--- if we start
close to zero and the trajectory converges, it should converge to a vector that
is close to an $\ell_1$-sparse solution of the underdetermined linear system
(assuming one exists).

Surprisingly, the convergence of both algorithms is not straightforward. For
\eqref{eq:md}, standard mirror descent analysis does not apply, even with a
fixed stepsize. Existing results~\cite{wu2020continuous,wu2021implicit} typically guarantee convergence only under restrictive conditions, such as very small or infinitesimal stepsizes, or stepsizes established via backtracking or linesearch~\cite{hurault2023convergent}.  

Our goal is to address this gap by introducing a simple adaptive stepsize rule with explicit convergence rates, all without restrictive assumptions. Along the way, we refine existing bounds on implicit bias. Interestingly, our analysis extends to arbitrary convex $L$-smooth functions with a known minimal value $f^*$. Additionally, we propose an alternative method that avoids exponentiation, resembling the original gradient descent with Hadamard overparametrization.  
In the next section, we provide the necessary background to explain our
motivation and the specifics of our approach.

\section{Background and motivation}

\begin{tcolorbox}[breakable, enhanced,boxrule=0pt,frame hidden,borderline west={1mm}{-2mm}{red}]
  \subsubsection*{\color{red} N.B.}
As the title suggests, we study entropic mirror descent for linear systems. However, it is immediately clear that \eqref{eq:md} cannot solve a general linear system~--- all its iterates are positive vectors. Instead, we primarily focus on \emph{nonnegative} linear systems of the form $A\bx = \bb$, with $\bx \in \R^n_{+}.$ Computationally, this is a harder problem due to the additional nonnegativity constraint. However, from a presentation perspective, it is much simpler, and thus we mostly focus on it. Later, we will show how to reduce a general linear system to the nonnegative setting.
\end{tcolorbox}
\paragraph{Implicit bias.}
Overparametrization in neural networks has led researchers to reconsider what drives an optimization algorithm to select one solution over another. One such direction is the study of \emph{implicit bias}, the tendency of an algorithm to converge to a specific solution, see  \cite{woodworth_kernel_2020,ji2019implicit,gunasekar2018characterizing,soudry2018implicit,even2023sgd} and references therein. Sometimes this bias is independent of user choices (e.g., in linearly separable datasets). Mostly, however, it is indirectly influenced by the user through the choice of an initial point, hyperparameters, etc. A well-known manifestation of this phenomenon is that gradient descent applied to the least squares problem converges to the solution closest in $\ell_2$-norm to the initial point $\bx_0$.

\paragraph{Hadamard overparametrization.} Multiplication of either vectors or matrices is the simplest operation that breaks convexity and which is so
omnipresent in neural networks, matrix factorization, phase retrieval, etc.
To start approaching these difficult problems, one can take a convex problem and
artificially introduce nonconvexity via vector/matrix multiplication,
see~\cite{zhao2019implicit,vaskevicius_implicit_2019,chou2025get,chou2023more,li2023simplex,poon2023smooth,ouyang2025kurdyka}
and study them instead.

In this paper, we consider one of the simplest problems~--- a nonnegative linear system
\begin{equation}
  \label{eq:problem}
  A\bx=\bb, \quad \bx\in \R^n_{+}.
\end{equation}
One can introduce the parametrization $\bx = \bu\circ \bu = \bu^2$, with $\bu\in \R^n$, and instead try
to study an unconstrained problem
\begin{align*}
A(\bu\circ \bu) = \bb 
\end{align*}
or its minimization analog
\begin{equation}
  \label{eq:12per1e}
  \min_{\bu} f(\bu\circ \bu) = \frac 12 \norm{A(\bu\circ   \bu)-\bb}^2.
\end{equation}
Applying gradient descent with the stepsize $\alpha$ to the problem above gives
\begin{equation}\label{eq:102hdn}
  \bu_{k+1} = \bu_k - 2\alpha  \nabla f(\bu_k\circ \bu_k)\circ \bu_k.
\end{equation}
Problem~\eqref{eq:12per1e} is nonconvex, so it is already interesting to
understand the convergence of $(\bu_k)$ to a solution: Is it a minimum? Is this
minimum local or global? On the other hand, we can study the implicit bias of
this scheme. This is also intriguing because, in practice, $(\bu_k)$ often
converges to a sparse solution of \eqref{eq:problem} when $\bu_0$ is initialized
close to $\zero$. Although convergence and implicit bias may appear independent,
the latter cannot exist without the former. Moreover, convergence has only been
established in restricted settings: for small stepsizes $\alpha$ or for the gradient
flow (with infinitesimal $\alpha$), as shown in~\cite{wu2021implicit,chou2023more}.

If we square equation~\eqref{eq:102hdn}, we get
\begin{align*}
  \bu_{k+1}^2 = \bu_k^2 \circ (\one - 2\alpha  \nabla f(\bu_k\circ \bu_k))^2,
\end{align*}
which can be rewritten again in terms of the $\bx$ variable as
\begin{equation}
  \label{eq:f322f}
  \bx_{k+1} = \bx_k\circ  (\one - 2\alpha  \nabla f(\bx_k))^2.
\end{equation}
Now, if $\alpha$ is small enough, then using $\exp(-t) \approx 1-t $ for
small $t$, we can approximate $\bx_{k+1}$ as 
\begin{align*}
 \bx_{k+1} = \bx_k\circ  (\one - 2\alpha  \nabla f(\bx_k))^2 \approx \bx_k \circ \exp (-4\a
  \nabla f(\bx_k)).
\end{align*}
The latter update can be interpreted as mirror descent with the (negative) entropy $h(\bx) = \lr{\bx,\log(\bx)-\one}$ as the kernel function. This connection has been used to establish a similar implicit bias in mirror descent.

\paragraph{Mirror descent.}
Let $D_h(\bx,\by) = h(\bx)-h(\by) - \lr{\nabla h(\by),\bx-\by}$ be the Bregman divergence
defined by $h$. Using it as a regularizer instead of the more common $\frac
12\norm{\cdot}^2$, yields the update of the mirror descent~(MD)
\begin{align}
  \label{eq:md_general}
  \bx_{k+1} = \argmin_{\bx}\lrset{ f(\bx_k) + \lr{\nabla f(\bx_k), \bx-\bx_k} + \frac{1}{\a}D_h(\bx,\bx_k)}.
\end{align}
If we apply \eqref{eq:md_general} to our setting with $f(\bx)=\frac12\norm{A\bx-\bb}^2$ and
$h(\bx)=\lr{\bx,\log(\bx)-\one}$, we obtain the update
\begin{equation}
  \label{eq:mdmd}
\bx_{k+1} = \bx_k \circ \exp(-\alpha  \nabla f(\bx_k)),  
\end{equation}
known also as entropic mirror descent. There are two main lines of establishing convergence of the mirror descent:
\begin{enumerate}[(i)]
\item\, Strong convexity of the function $h$, see~\cite{beck2017first}.
  
\item\, Relative smoothness of $f$ wrt $h$, that is the inequality $D_f(\bx,\by)\leq L
  D_h(\bx,\by)$ must hold for all $\bx,\by$ and some $L>0$, see~\cite{Bauschke2016,lu2018relatively}.
\end{enumerate}
Roughly speaking, for our choice of $f$ and $h$, both conditions are equivalent
to $D_h$ being lower-bounded by a quadratic function. However, this cannot hold,
as $h$ is defined over $\R^n_+$ and not the unit simplex where it is strongly
convex. Furthermore, we can state an even stronger result, with the proof (and
precise definition) provided in \Cref{app:unstable}.
\begin{proposition}[Instability of mirror descent with constant stepsizes independent of $\bb$]\label{prop:diverge}
  For any $A \neq 0 \in \R^{m \times n}$ and a stepsize $\alpha > 0$, there exists a vector $\bb \in \R^m$ such that, for mirror descent~\eqref{eq:mdmd} applied to $f(\bx) = \frac{1}{2} \norm{A\bx - \bb}^2$, any solution is an unstable fixed point.
\end{proposition}

One conclusion we can make from \Cref{prop:diverge} is that either (i) stepsize
$\alpha$ has to be adaptive  or (ii) $\alpha$ must depend on $\bb$ as well, which makes it
difficult to come up with a good rule for it.
 
In a sense, we have a somewhat paradoxical situation: the simplest setting~--- a
quadratic objective and the most popular Bregman kernel, negative entropy~---
and yet there is no obvious way to establish convergence of mirror descent (at
least with an $O(1/k)$ rate).

Entropic mirror descent is also widely used in the theoretical computer science literature, where it is better known as the Hedge algorithm~--- a variation of the celebrated Multiplicative Weights Update. The original analysis of the Hedge algorithm, due to~\cite{freund1997decision}, employed a different approach based on the quadratic approximation of the exponential mapping. This idea will also play a key role in our analysis.

\paragraph{Polyak's stepsize.}
Initially proposed for nonsmooth optimization, the Polyak stepsize
\cite{polyak1969minimization} remains a popular choice even in the smooth setting when the optimal function value $f^*$ is known: it is simple to use and often performs well in practice~\cite{loizou2021stochastic,hazan2022revisiting,wang2023generalized,jiang2024adaptive}. Its original interpretation for the (sub)gradient method, $\bx_{k+1} = \bx_k - \alpha_k \nabla f(\bx_k)$, relies on a linear approximation of $f$ by its local model $\tilde f$:
\begin{align*}
\tilde f(\bx) = f(\bx_k) + \lr{\nabla f(\bx_k), \bx - \bx_k}, 
\end{align*}
and sets $\alpha_k$ such that $\tilde f(\bx_{k+1}) = f^*$. For the
update~\eqref{eq:mdmd}, applying the same rule does not yield an explicit
stepsize. Instead, we use an approximation based on $\exp(-t) \approx 1 - t$:
\begin{align*}
f(\bx_k) + \lr{\nabla f(\bx_k), \bx_k \circ \exp(-\alpha_k \nabla f(\bx_k)) - \bx_k} \approx f(\bx_k) - \alpha_k \lr{\nabla f(\bx_k), \bx_k \circ \nabla f(\bx_k)} \stackrel{?}{=} f^*. 
\end{align*}
This provides a slightly simplified description of our approach in the paper, and justifying this approximation is a central part of the convergence proof.

\paragraph{Roadmap.}
\Cref{sec:nls} focuses on a nonnegative linear system, analyzing the implicit bias of MD and its convergence properties with Polyak's stepsizes. 
\Cref{sec:gen} explores several generalizations, including the extension of these results to general linear systems, an alternative to Hadamard descent with provable guarantees and the extension of entropic MD to arbitrary convex $L$-smooth functions. Finally, additional results, proofs, and numerical experiments are provided in the Appendix.

\paragraph{Notation.} 
We use $\R^n_{+}$ and $\R^n_{++}$ to denote the nonnegative and strictly positive orthants, respectively. Similarly, $S$ and $S_+$ represent the solution sets of a general linear system and a nonnegative linear system. Unless stated otherwise, all vector operations are elementwise. We use $\circ$ to denote Hadamard (elementwise) multiplication. The all-zero and all-one vectors are denoted by $\zero$ and $\one$. For $\bx\in \R^n$, $x_i$ denotes the $i$-th component and $x_{\min}$ denotes the minimum over all components of the vector $\bx$. Whenever the function $h$ is mentioned, it always refers to the entropy function $h(\bx) = \lr{\bx, \log (\bx)-\one}$.

\section{Nonnegative linear system}\label{sec:nls}

\subsection{Implicit bias of entropic mirror descent}
In this subsection, we give an overview of the main results explaining the implicit bias of mirror descent, and why initialization close to the origin  yields approximate $\ell_1$-norm minimization. 

If we applied gradient descent to $\min_{\bx} f(\bx) = \frac{1}{2} \norm{A\bx - \bb}^2$, its implicit bias would imply that the trajectory $(\bx_k)$ satisfies $\bx_k \to P_{S} \bx_0$, where $P_{S}$ denotes the metric projection. This follows immediately because the iterates of gradient descent satisfy $\bx_k - \bx_0 \in \range(A^\top)$, and $\range(A^\top) \perp \ker(A)$. For mirror descent, the situation is analogous, except that the metric projection is replaced by the Bregman projection and $S$ by $S_+$. This statement is well known, see e.g. \cite{gunasekar2018characterizing, wu2021implicit}.
\begin{proposition}\label{prop:implicitbias}
  If the MD iterates, $(\bx_{k})$, converge to $\bx^{*}\in S_+$, then we have 
  $\bx^* = \argmin_{\bx\in S_{+}}D_h(\bx, \bx_{0}).$
  Equivalently, $\bx^{*}\in S_+$ satisfies
  \begin{align}\label{eq:k23d1}
  \lr{\nabla h(\bx^*)-\nabla h(\bx_0), \bz - \bx^*} = 0\quad\forall\,\bz \in S_+.
  \end{align}
\end{proposition}
This characterization using the Bregman projection can be used to establish further results on the implicit bias of mirror descent. 
For example, with initialization chosen near the origin $0$, it is well known that mirror descent has an implicit bias towards $\ell_1$-sparse solutions \cite{wu2021implicit, chou2023more, gunasekar2017implicit}. We will take a closer look at this implicit bias and discuss the two main approaches used for quantifying this behavior.
In the following section we let $\bx$ denote $\argmin_{\bx^\prime\in S_+} D_h(\bx^\prime,\bx_0)$ i.e. the limit of MD and let $\bz$ be an $\ell_1$-sparse (or $\ell_1$-minimal) solution, that is, a solution of the nonnegative basis pursuit problem
\begin{equation}
  \label{eq:bp}
\min \norm{\bx}_1 \quad \text{subject to}\quad A\bx = \bb,~~ \bx\in \R^n_{+}.   
\end{equation}.

\subsubsection{Slow rates}
A common approach \cite{wu2021implicit, chou2023more} derives worst case bounds for a fixed initialization. The resulting bounds  are explicit, requiring little knowledge about the solution space of the linear system. One such result is a special case of \cite[Theorem 2.1]{chou2023more} rewritten here to fit our setting.
\begin{proposition}{\cite[Theorem 4]{chou2023more}}\label{cor:bound1}
Let $\bx_0=e^{-\eta}\one$ with $\norm{\bx_0}_1<\norm{\bz}_1$. Then
\begin{align} \label{eq:chou2023more}
 \norm{\bx}_1-\norm{\bz}_1\leq \norm{\bz}_1\frac{\log n}{\eta+\log\frac{\norm{\bz}_1}{n}}.
 \end{align}
\end{proposition}
While this result yields concrete sparsity bounds, they are quite conservative in practice and do not fully explain the observed sparsity. One way of establishing this bound follows directly from  \Cref{prop:implicitbias}. Rearranging equation \eqref{eq:k23d1} when assuming  $\norm{\bx_0}_1<\norm{\bx}_1$ yields
 \begin{align} \label{eq:rearangedIB}
 \norm{\bx}_1-\norm{\bz}_1= \norm{\bz}_1\frac{\inner{\log \tilde\bx}{ \tilde\bz}-\inner{\log {\tilde\bx}}{ \tilde\bx}}{\eta+\log\norm{\bx}_1+\inner{\log {\tilde\bx}}{ \tilde\bx}},
 \end{align}
 where  $\tilde \bx=\frac{\bx}{\norm{\bx}_1}$ and $\tilde \bz=\frac{\bz}{\norm{\bz}_1}$, see \Cref{app:implicitbias} for details. 
\Cref{cor:bound1} now follows by noting that $\inner{\log {\tilde\bx}}{ \tilde\bx}\geq-\log(n)$ and $\inner{\log \tilde\bx}{ \tilde\bz} \leq0$. While this approximation may seem rather crude, especially since in the numerator $\inner{\log {\tilde\bx}}{ \tilde\bx}=-\log n$ and $\inner{\log \tilde\bx}{ \tilde\bz}=0$ cannot hold simultaneously, the resulting bound cannot be improved by much. By optimizing over the numerator, rather than bounding both inner products separately, we can establish both a general upper bound that is tighter than \eqref{eq:chou2023more} and get a lower bound on the worst case.
\begin{proposition}\label{pro:imprbound1}
Let $\bx_0=e^{-\eta}\one$ with $\norm{\bx_0}_1<\norm{\bz}_1$. Then
\begin{align*} 
 \norm{\bx}_1-\norm{\bz}_1\leq \norm{\bz}_1\frac{W_0(\frac{n-1}{e})}{\eta+\log\frac{\norm{\bx}_1}{n}},
 \end{align*}where $W_0$ denotes the Lambert-$W$ function. 
This bound is nearly sharp: i.e, there exist a linear system for which
\begin{align*} 
\norm{\bx}_1-\norm{\bz}_1\geq \norm{\bz}_1\frac{W_0(\frac{n-1}{e})}{\eta+\log\frac{\norm{\bx}_1}{n}+1}.
 \end{align*}
\end{proposition}
For a proof, see \Cref{app:implicitbias}. However, despite being near tight, this refinement has the same asymptotic scaling in both $\eta$ and $n$ as the bound \eqref{eq:chou2023more}, because $W_0(n)\approx\log(n)-\log(\log(n))$. This highlights that the obtained slow rate is not an artifact of a crude step in the analysis but instead intrinsic to this kind of bound. The main culprit here is that the instances attaining these worst cases depend not only on the dimension but also on the initialization itself.
\subsubsection{Fast rates}
Much sharper bounds can be established when first fixing a problem instance and then analyzing the decay of the $\ell_1$-gap as the initialization tends to zero. The main observation here is that the limit of MD is the solution of a linear program with an entropic penalty term. Let $\bx^\eta$ be the solution of $\argmin_{\bx\in S_+}D_h(\bx,\bx_0)$ with $\bx_0=e^{-\eta}\one$. Then
\begin{align*}
\bx^\eta=&\argmin_{\bx\in S_+}\inner{\bx}{\log\bx-\one}+\eta\inner{\bx}{\one}+ne^{-\eta}\\
&=\argmin_{\bx\in S_+} \inner{\one}{\bx}+\frac{1}{\eta}\inner{\bx}{\log\bx-\one}. 
\end{align*}

In this setting results from linear programming literature tell us that $\norm{\bx^\eta}_1-\norm{\bz}_1$ can actually converge linearly as $\eta$ tends to infinity. In particular, under the assumptions that $A$ has full rank and that $A\bx=\bb$ admits a strictly positive solution, Cominetti and San Martín \cite{cominetti1994asymptotic} prove that for $\bz\in\argmin_{\bx\in S_+} \inner{\one}{\bx}$ one has $\bx^\eta=\bz+O(e^{-M\eta})$ for some $M>0$ as $\eta\to\infty$.
Naturally, this also proves that $\norm{\bx^\eta}_1-\norm{\bz}_1=O(e^{-M\eta})$ and provides intuition as to why in practice one observes much sparser solutions than what the crude bound in \Cref{cor:bound1} would suggest. 

The main limitation of this approach is that it is hard to get concrete bounds on the rate of the linear convergence. Convergence rates for entropic regularized linear programs are established, when $S_+$ is bounded, see \cite{weed2018explicit}, or if $S_+$ is a subset of the simplex \cite{muller2025fisher}. However, these bounds require strong knowledge of the geometry of the solution space, which is usually unknown.

\subsection{Algorithm}
So far, whenever we have discussed implicit bias, we have included the condition
``if the trajectory converges.'' In this section, we address this issue.

For $\bv \in\R^n$, define $\norm{\bv}^2_{\bx}:=\inner{\bx}{\bv^2}$.  The algorithm we propose is a combination of the Polyak-type stepsize with mirror descent
\begin{equation}
  \label{eq:alg}
\begin{aligned}
\alpha_k &= \min\left\{\frac{f(\bx_k)}{\norm{\nabla f(\bx_k)}^2_{\bx_k}},\frac{1.79}{\norm{\nabla f(\bx_k)}_\infty}\right\}\\
 \bx_{k+1}&=\bx_{k}\circ\exp(-\alpha_k\nabla f(\bx_k)).
\end{aligned}
\end{equation}

\begin{theorem}\label{thm:convergence}
Let $A\in \R^{m\times n}$ and $\bb\in \R^m$ such that $S_{+}:=\{\bz\in \R_{+}^n
: A\bz=\bb\}$ is nonempty. Then Algorithm~\eqref{eq:alg} converges to a solution $\bx^*\in S_+$ and
\begin{align*}
 \min_{i\leq k}f(\bx_i)\leq \frac{4R(R+\norm{\bx^*}_1)\max_{j\leq n}\norm{A_{:j}}_2^2}{k+1},
\end{align*} where $R=D_h(\bx^*,\bx_{0})$.
\end{theorem}
For the proof, we require several auxiliary lemmas. Their proofs are short and elementary, and they can be skipped entirely upon the first reading.
 \begin{lemma}[Folklore]
 \label{lem:phi_positve_increasing}
 If $t\leq 1.79$, then $\exp(t)\leq 1+t+t^2$.
\end{lemma}
\begin{proof}
 Let $\phi(x) = 1 + x + x^2 - e^x$. We want to show that $\phi(x) \ge 0$ for all $x \le 1.79$. Note that $\phi$ is convex on $(-\infty, \log 2)$ and concave on $(\log 2, \infty)$. Since $\phi(0) = \phi'(0) = 0$ convexity implies that $\phi(x) \ge 0$ for all $x \le \log 2$. Finally, given that both $\phi(\log 2) \ge 0$ and $\phi(1.79) \ge 0$, concavity implies that $\phi(x) \ge 0$ for all $x \in [\log 2, 1.79]$.
\end{proof}

The statement below can be seen as a generalization of the fact that entropy is strongly convex over the unit simplex. This inequality is likely known, but we have not been able to find it in the literature.
\begin{lemma}[Generalized Pinsker's inequality]\label{lem:generalizedPinsker}
Let $h(\bx)=\lr{\bx,\log (\bx)-\one}$. Then for any  $\bx\in \R^n_{+}$ and $\by\in \R^n_{++}$ it holds that
  \begin{equation}
    \label{eq:pinsker_gen}
    D_h(\bx,\by) \geq \frac 12 \frac{\norm{\bx-\by}_1^2}{\max\{\norm{\bx}_1, \norm{\by}_1\}}.
  \end{equation}
\end{lemma}
\begin{proof}
  By the mean-value theorem,
  \begin{align*}
   D_h(\bx,\by) = h(\bx) - h(\by) - \inner{\nabla h(\by)}{\bx-\by} = \frac 12 \inner{\nabla^2
      h(\boldsymbol{\xi})(\bx-\by)}{\bx-\by},
  \end{align*}
  where $\nabla^2 h(\bx) = \diag(1/\bx)$ and $\boldsymbol{\xi}\in [\bx,\by]$. Let $\boldsymbol{\xi}=t\bx+(1-t)\by$ for
  some $t\in (0,1)$. Then we can bound
  $2D_h(\bx,\by)$ as
  \begin{flalign*}
    2D_h(\bx,\by) &= \sum_{i=1}^n \frac{(\bx_i-\by_i)^2}{\boldsymbol{\xi}_i}\geq \frac{\left(\sum_{i=1}^n
              \abs{\bx_i-\by_i}\right)^2}{\sum_{i=1}^n \boldsymbol{\xi}_i} &  \hspace{-2cm} \text{(Cauchy-Schwarz)}\\ &=
                                                           \frac{\norm{\bx-\by}_1^2}{\sum_{i=1}^n(t\bx_i+(1-t)\by_i)}=\frac{\norm{\bx-\by}_1^2}{t\norm{\bx}_1+(1-t)\norm{\by}_1} \geq \frac{\norm{\bx-\by}_1^2}{\max\{\norm{\bx}_1, \norm{\by}_1\}}.\qedhere
  \end{flalign*}
\end{proof}

\begin{lemma}\label{lem:genPinskerbound}
  For any  $x\in \R^n_{+}$ and $y\in \R^n_{++}$ it holds that
  \begin{equation}
    \label{eq:max_bound}
    \max\{\norm{\bx}_1,\norm{\by}_1\} \leq 2D_h(\bx,\by) + 2\min\{\norm{\bx}_1,\norm{\by}_1\}.
  \end{equation}
\end{lemma}
\begin{proof}
  This is a direct consequence of the previous lemma. Let
  $ M= \max\{\norm{\bx}_1,\norm{\by}_1\}$ and $m =
  \min\{\norm{\bx}_1,\norm{\by}_1\}$. Applying the triangle inequality
  $\norm{\bx-\by}_1\geq M - m$ in \eqref{eq:pinsker_gen}, we get
\begin{align*}
  2D_h(\bx,\by)\geq \frac{(M-m)^2}{M} = M - 2m + \frac{m^2}{M}\geq M - 2m,
\end{align*}
and the desired inequality follows.
\end{proof}
\begin{proof}[Proof of \Cref{thm:convergence}]
The proof proceeds in three steps. First, we show that \textbf{(a)} the step size in \eqref{eq:alg} allows us to bound the descent of the Bregman divergence as  
\begin{align*}
 D_h(\bz,\bx_{k+1}) - D_h(\bz,\bx_{k}) \leq -\alpha_k f(\bx_k),
\end{align*}
where $\bz\in S_+$. Telescoping this inequality gives  
\begin{align}\label{eq:32rf3q}
 \sum_{i=0}^k \alpha_i f(\bx_{i})  \leq D_h(\bz,\bx_{0}).
\end{align}  
It then remains to show that \textbf{(b)} the step sizes $(\alpha_k)$ are bounded away from zero to ensure an $O(1/k)$ rate. Finally, this will lead to \textbf{(c)} the convergence of iterates.

\textbf{(a)} By the Three-Points-Lemma \cite[Chapter~9.2, p.~252]{beck2017first} of mirror descent for all iterates
 $(\bx_{k})\subset \R^n_{++}$ it holds that 
\begin{align}\label{eq:r1rfcq}
 D_h(\bz,\bx_{k+1})-D_h(\bz,\bx_{k})= -\alpha_k\langle\nabla f(\bx_k),\bx_{k}-z\rangle+ D_h(\bx_{k},\bx_{k+1}).
\end{align}
For $f(\bx)=\frac{1}{2}\norm{A\bx-\bb}^2$ and  $\bz\in S_+$ we have
\begin{align}\label{eq:23qfq2}
 \alpha_k\inner{\nabla f(\bx_k)}{\bx_{k}-\bz}=\alpha_k\inner{A^\top(A\bx_k-\bb)}{\bx_{k}-\bz} = 2\alpha_k f(\bx_k).
\end{align}
By the entropic mirror descent update rule \eqref{eq:md}, it follows that  $D_h(\bx_{k},\bx_{k+1})$ takes the explicit form
\begin{align}
D_h(\bx_{k},\bx_{k+1})&=\inner{\bx_{k}}{\log \bx_{k}/\bx_{k+1}}+\inner{\bx_{k+1}-\bx_{k}}{\one}\nonumber\\
&=\inner{\bx_{k}}{\log \bx_{k}/\bx_{k+1}-\one+\bx_{k+1}/\bx_{k}}\nonumber\\
&=\inner{\bx_{k}}{\alpha_k\nabla f(\bx_{k})-\one+\exp{\left(-\alpha_k\nabla f(\bx_{k})\right)}}.\label{eq3:s-1+exp}
\end{align}
We want to find an upper bound for the right-hand side of \eqref{eq3:s-1+exp}.
In~\Cref{lem:phi_positve_increasing} we established that $\exp(t)\leq 1+t+t^2$ for $t\leq1.79$. 
Consequently, since by definition $\alpha_k\norm{\nabla f(\bx_k)}_\infty
\leq1.79$ and $\bx_k\in \R^n_{++}$, we have that 
\begin{align*}
D_h(\bx_{k},\bx_{k+1})\leq\inner{ \bx_k}{\alpha_k \nabla f(\bx_k) -\one + (\one-\alpha_k \nabla f(\bx_k)+\alpha_k^2 \nabla f(\bx_k)^2)}= \inner{ \bx_k}{\alpha_k^2 \nabla f(\bx_k)^2}.
\end{align*}
The term $\inner{ \bx_k}{\nabla f(\bx_k)^2}$ can be written as $\norm{\nabla f(\bx_k)}^2_{\bx_k}$, the square of a weighted gradient norm.
Finally, going back to~\eqref{eq:r1rfcq}, we conclude
\begin{align}
 D_h(\bz,\bx_{k+1})-D_h(\bz,\bx_{k})&\leq - 2 \alpha_k f(\bx_k)+\alpha_k^2\norm{\nabla f(\bx_k)}^2_{\bx_k}.\label{eq2:after_exp_1+x+sx2_bound}\\
 &=-\alpha_k f(\bx_k)+\alpha_k(\alpha_k\norm{\nabla f(\bx_k)}^2_{\bx_k}-f(\bx_k))\label{eq2:line2}\\
 &\leq-\alpha_k f(\bx_k),\nonumber
\end{align}
where in \eqref{eq2:line2} we used that $\alpha_k\norm{\nabla f(\bx_k)}^2_{\bx_k}\leq f(\bx_k)$. Note that $\alpha_k$
actually minimizes the right-hand side of expression
\eqref{eq2:after_exp_1+x+sx2_bound}, given the constraint $\alpha_k\leq
\frac{1.79}{\norm{\nabla f(\bx_k)}_\infty}$. Thus, from this perspective
$\alpha_k$ closely resembles the traditional Polyak stepsize for gradient
descent. This proves~(a).

\textbf{(b)} Here we  show that the sequence $(\alpha_k)$ is strictly separated from $0$. Let $\norm{A}_{p,q}:= \sup_{\bx\in \R^n} \frac{\norm{A x}_q}{\norm{\bx}_p}$.  By Hölder's inequality, we have
\begin{align*}
\frac{f(\bx_k)}{\norm{\nabla f(\bx_k)}_{\bx_k}^2} &\geq  \frac{f(\bx_k)}{\norm{\bx_k}_1\norm{\nabla f(\bx_k)}_\infty^2}\geq \frac{1}{2\norm{\bx_k}_1\norm{A^\top}_{2,\infty}^2}
\end{align*}  and
\begin{align*}
 \frac{1.79}{\norm{\nabla f(\bx_k)}_\infty}&\geq\frac{1.79}{\norm{A^\top A}_{1,\infty}\norm{\bx_k-\bz}_1}.
\end{align*} 
Note that $\norm{A^\top}_{2,\infty}$ is given by $\max_{j\leq n}\norm{A_{:j}}_2$, the maximum column $\ell_2$-norm of the matrix $A$. Moreover, the operator norm $\norm{A^\top A}_{1,\infty}$ is given by the largest entry of $A^\top A$ that is $\max_{i,j\leq n}\inner{A_{:i}}{A_{:j}}=\max_{j\leq n}\norm{A_{:j}}_2^2$. 

Combining the above inequalities, we can bound $\alpha_k$ by
\begin{align*}
 \alpha_k\geq \max_{j\leq n}\norm{A_{:j}}_2^{-2}\left(\max\left\{\norm{\bx_k-\bz}_{1},2\norm{\bx_k}_1\right\}\right)^{-1}.
 \end{align*}
By~\Cref{lem:genPinskerbound}, we have that 
\begin{align*}
 \norm{\bx_k}_1\leq  2D_h(\bz,\bx_k) + 2\norm{\bz}_1
\end{align*}
and by the triangle inequality,
\begin{align*}
\norm{\bz-\bx_k}_1\leq \norm{\bz}_1+\norm{\bx_k}_1\leq 2D_h(\bz,\bx_k) + 3\norm{\bz}_1.
\end{align*}
Since the sequence $D_h(\bz,\bx_k)$ is nonincreasing, we deduce that
\begin{align*}
\max\{2\norm{\bx_k}_1, \norm{\bx_k-\bz}_1\}\leq 4(D_h(\bz,\bx_k) + \norm{\bz}_1)\leq 4(D_h(\bz,\bx_0) + \norm{\bz}_1).  
\end{align*}
Thus, 
\begin{align*}
 \alpha_k\geq \frac{1}{4(D_h(\bz,\bx_0) + \norm{\bz}_1)\max_{j\leq n}\norm{A_{:j}}_2^2},
\end{align*} which proves~(b). Now, using this estimate in \eqref{eq:32rf3q}, we obtain
\begin{align*}
 \min_{0\leq i\leq k}f(\bx_i)\leq \frac{4D_h(\bz,\bx_{0})(D_h(\bz,\bx_0) +\norm{\bz}_1)\max_{j\leq n}\norm{A_{:j}}_2^2}{k+1}.
\end{align*}

\textbf{(c)} 
To prove the convergence of the iterates, we first note that the sequence $(\bx_k)$ is bounded (as proved in (b)). Consequently, there exists a convergent subsequence $\bx_{t_k} \to \bx^\infty$. By continuity, we have
\begin{align*}
f(\bx^\infty) = \lim_{k\to\infty} f(\bx_{t_k}) = 0,
\end{align*}
which implies that $\bx^\infty \in S_+$. Thus, the sequence $D_h(\bx^\infty,\bx_k)$ is decreasing and hence convergent. Since $D_h(\bx^\infty,\bx_{t_k}) \to 0$ along the subsequence (due to $\bx_{t_k} \to \bx^\infty$), it follows that $D_h(\bx^\infty,\bx_k)$ converges to $0$. Finally, from~\Cref{lem:generalizedPinsker}, we know that
\begin{align*}
2D_h(\bx^{\infty},\bx_k) \max(\norm{\bx^\infty}_1, \norm{\bx_k}_1) \geq \norm{\bx^\infty - \bx_k}_1^2, 
\end{align*}
which ensures that $\bx_k \to \bx^{\infty}$. The proof is complete.
\end{proof}

\begin{remark}
  The main crux of the proof is a quadratic approximation of the exponential
  function: $\exp(t) \leq 1 + t + t^2$ for $t \leq 1.79$. Notably, the Hedge
  Algorithm~\cite{freund1997decision} also relies on a similar quadratic bound
  and can be viewed as an entropic mirror descent, albeit over the unit
  simplex. Our approach differs in that we analyze $D_h$ rather than the
  $\ell_1$-norm of the update.
\end{remark}

\begin{remark}
 An alternative stepsize strategy that can be used to establish convergence is backtracking. Specifically, by choosing the stepsize $\alpha_k$ such that 
$\alpha_k D_f(\bx_k,\bx_{k+1}) < D_h(\bx_k,\bx_{k+1})$, 
one can derive sublinear convergence rates using an analysis similar to that of \cite{lu2018relatively}. 
However, this approach has several drawbacks.  Not only does each iteration become more expensive, but our simulations show that, on average, the step sizes produced by our version of Polyak's stepsize are larger than those from backtracking. As a consequence, our method is faster in practice, even in terms of iteration count. See Figure~\ref{fig:experiment10}. This is not surprising, since the condition $\alpha_k D_f(\bx_k,\bx_{k+1})<D_h(\bx_k,\bx_{k+1})$ ensures that $f(\bx_k)$ is non-increasing which is a stricter requirement than the one imposed by our method.
\end{remark}

\subsection{Linear convergence}\label{subsec:linconv}
 A natural question is whether the sublinear rate of entropic mirror descent can be improved. In  \Cref{thm:linear_rate}, we establish global and local linear convergence rates in the case when the solution $\bz= \argmin_{\bx\in S_{+}}D_h(\bx, \bx_{0})$ is strictly separated from $0$ in every variable. Note that this rate depends on $ z_{\min}>0$ and thus does not guarantee linear rates when converging to sparse solutions.

\begin{theorem}\label{thm:linear_rate}
Let $(\bx_k)$ be as in \Cref{thm:convergence}. Suppose that $\bx_k$ converges to a solution $\bz$ that is strictly separated from the boundary of the nonnegative orthant, i.e., $ z_{\min}>0$. Then 
\begin{align*}
  D_h(\bz,\bx_{k+1})\leq D_h(\bz,\bx_{k})\left(1-\frac{ -\lambda_{\min}^{+} z_{\min}W_0\left(-\exp\left(-1-\tfrac{D_h(\bz,\bx_k)}{ z_{\min}}\right)\right)}{8\max_{j\leq n}\norm{A_{:j}}_2^2\left(\norm{\bz}_1+D_h(\bz,\bx_k)\right)}\right),
\end{align*} where $\lambda_{\min}^{+}$ is the smallest positive eigenvalue of $A^\top A$. Furthermore, replacing $D_h(\bz,\bx_k)$ with $D_h(\bz,\bx_0)$ in the above fraction yields a global linear rate. This global rate can be very slow, but we also get a local linear rate of 
\begin{align*}
 1-\frac{\lambda_{\min}^{+}  z_{\min}}{8\max_{j\leq n}\norm{A_{:j}}_2^2\norm{\bz}_1}.
\end{align*}
\end{theorem}
The proof of \Cref{thm:linear_rate} primarily hinges on the Polyak stepsize being bounded below by a positive constant and can be found in \Cref{app:linear_rate}. As a result, the analysis also establishes a local linear convergence rate for any constant stepsize smaller than this bound. Moreover, the local linear rate can be improved to 
\begin{align*}
 1-\frac{\lambda_{\min}^{+}  z_{\min}}{2\max_{j\leq n}\norm{A_{:j}}_2^2\norm{\bz}_1}.
\end{align*}  by employing a tighter bound in place of \Cref{lem:genPinskerbound} and in place of $D_h(\bs,\bx) x_{\min} \leq \norm{\bx - \bs}^2$  when $\bs$ and $\bx$ are sufficiently close.

 \section{Generalization}\label{sec:gen}

\subsection{Linear system}\label{sec:ls}
 We now consider the general linear system $A\bx=\bb$ with $\bx\in\R^n$. Rewriting $\bx=\bu-\bv$ with $\bu,\bv\in\R_+^n$ reduces the problem to a nonnegative setting. Define $f(\bx)=\frac{1}{2}\norm{A\bx-\bb}^2$. Since $\nabla_{\bu} f(\bu-\bv)=\nabla f(\bu-\bv)$ and $\nabla_{\bv} f(\bu-\bv)=-\nabla f(\bu-\bv)$, entropic mirror descent yields the updates
\begin{align*}
\bu_{k+1} &= \bu_k\circ\exp\Bigl(-\alpha_k\,\nabla f(\bu_k-\bv_k)\Bigr), \\
\bv_{k+1} &= \bv_k\circ \exp\Bigl(\alpha_k\,\nabla f(\bu_k-\bv_k)\Bigr).
\end{align*}
This procedure is known as the EG\textpm\ algorithm and is equivalent to mirror
descent with a hyperentropy mirror map~\cite{ghai2020exponentiated}. Next, we
show that it converges with the Polyak stepsize defined as before.
\begin{corollary}\label{cor:EGpm}
Let $A\in\R^{m\times n}$ and $\bb\in\R^m$ be such that $S=\{\bx\in\R^n: A\bx=\bb\}\neq\varnothing$. Then, the EG\textpm\ algorithm with stepsize 
\begin{align*}
\alpha_k=\min\Biggl\{\frac{f(\bx_k)}{\norm{\nabla f(\bx_k)}_{\bu_k+\bv_k}^2},\;\frac{1.79}{\norm{\nabla f(\bx_k)}_\infty}\Biggr\},
\end{align*}
where $\bx_k=\bu_k-\bv_k$, generates iterates that converge to a solution $\bx^*=\bu^*-\bv^*\in S$. Furthermore,
\begin{align*}
 \min_{0\leq i\leq k}f(\bx_i)\leq \frac{4R(R+\norm{\bu^*+\bv^*}_1)\max_{j\leq n}\norm{A_{:j}}_2^2}{k+1},
\end{align*} where $R=D_h(\bu^*,\bu_{0})+D_h(\bv^*,\bv_{0})$.
\end{corollary}
\begin{proof}
Note that the EG\textpm\ algorithm here is equivalent to applying entropic mirror descent to the nonnegative system
\begin{align*}
\tilde{A}\bw &= \bb,\quad \bw=(\bu,\bv)\in\R_+^{2n},\quad \tilde{A}=(A,-A).
\end{align*}
The result then follows directly from \Cref{thm:convergence}.
\end{proof}

\subsection{Hadamard descent+}\label{subsec:hadamard+}
As we have already mentioned, the original motivation to study implicit bias of
MD for linear systems comes from considering Hadamard overparametrization. Unfortunately,
we don't know how to prove a global convergence of such an algorithm with the
proposed stepsize $\alpha_k$ (although numerical experiments suggest that it performs well; see \Cref{app:numerics}). Instead, we can suggest a
simple modification that will provably work:
\begin{equation}\label{eq:bigHadupdate}
\begin{aligned}
  \alpha_k &= \min\left\{\frac{f(\bx_k)}{\norm{\nabla f(\bx_k)}^2_{\bx_k}},\frac{1.79}{\norm{\nabla f(\bx_k)}_\infty}\right\}\\
  \bx_{k+1}&=\bx_k\circ (\one-\alpha_k\nabla f(\bx_k)+\alpha_k^2 \nabla f(\bx_k)^2)
\end{aligned}
\end{equation}
One can see that this scheme closely resembles \eqref{eq:f322f} which can be
rewritten as
\begin{equation*}
  \bx_{k+1} = \bx_k\circ  (\one - 4\alpha_k \nabla f(\bx_k) + 4\alpha_k^2 \nabla f(\bx_k)^2).
\end{equation*}
We still don't have a clear explanation of the intuition behind scheme~\eqref{eq:bigHadupdate}, other than it has all the properties of MD required for the proof in \Cref{app:Hadamard+}.
\begin{theorem}\label{thm:HDplusconvergence}
Let $A \in \R^{m \times n}$ and $\bb \in \R^m$ such that $S_{+}\coloneqq\{\bz \in \R_{+}^n : A\bz = \bb\}$ is nonempty. Then Algorithm~\ref{eq:bigHadupdate} converges to a solution $\bz \in S_+$. Furthermore, the sublinear convergence rate established for mirror descent in \Cref{thm:convergence} also holds for this algorithm.
\end{theorem}

\subsection{Entropic mirror descent for arbitrary convex functions}
  The schemes \eqref{eq:alg}  and \eqref{eq:bigHadupdate} can be used to solve
  $\min_{\bx \in \R^n_{+}} f(\bx)$ for any convex $L$-smooth function $f$ with a known optimal value $f^* = f(\bx^*)$. Specifically, we get convergence with the stepsize
  \begin{align*}
    \alpha_k&=\min\Biggl\{\frac{f(\bx_k)-f^*}{2\norm{\nabla
    f(\bx_k)}_{\bx_k}^2},\;\frac{1.79}{\norm{\nabla f(\bx_k)}_\infty}\Biggr\}.
\end{align*} The proof follows almost the same lines as in
\Cref{thm:convergence} and \Cref{thm:HDplusconvergence}. The only difference is that, instead of relying on the fact that $f$ is quadratic as in \eqref{eq:23qfq2}, we use convexity. This results in a slightly smaller stepsize. To ensure the steps are sufficiently large, we need to assume $L$-smoothness (or any other condition that guarantees it), see \Cref{sec:Proof_of_lower_boundedness_L_smooth_case}.

\section{Future directions}
There are many directions for further study: accelerated versions of mirror descent or its stochastic variants. From the perspective of implicit bias, one could also consider noisy systems where $f^*$ is unavailable and explore how to adapt the algorithm to such settings. Another natural extension is to generalize linear systems to matrix or tensor systems, known as matrix sensing. In this context, the scheme~\eqref{eq:bigHadupdate} is particularly appealing, as it avoids the need for matrix exponentiation.

\section*{Funding}
This research was funded in whole or in part by the Austrian Science Fund (FWF) [10.55776/STA223].

\bibliographystyle{abbrv}
\bibliography{joint_biblio.bib}
\appendix
\renewcommand*\theHsection{app.\Alph{section}}
\renewcommand*\theHsubsection{app.\Alph{section}.\arabic{subsection}} 

{
\crefalias{section}{appendix} 
\section*{Appendices}  
\section{Proof of implicit bias results}\label{app:implicitbias}
We start by proving the statement that rearranging equation \eqref{eq:k23d1} can yield information about the difference of $\ell_1$-norms.
\begin{lemma}\label{lem:imprbound1}
 Let $\bx_0=e^{-\eta}\one$, and $\bz$ be a solution of \eqref{eq:bp}. If iterates $(\bx_k)$ of Algorithm~\ref{eq:md} converge to a solution $\bx\in S_+$ with $\norm{\bx}_1> \norm{\bx_0}_1$, then it holds that 
  \begin{align} 
 \norm{\bx}_1-\norm{\bz}_1= \norm{\bz}_1\frac{\inner{\log \tilde\bx}{ \tilde\bz}-\inner{\log {\tilde\bx}}{ \tilde\bx}}{\eta+\log\norm{\bx}_1+\inner{\log {\tilde\bx}}{ \tilde\bx}},
 \end{align}
 where  $\tilde \bx=\frac{\bx}{\norm{\bx}_1}$ and $\tilde \bz=\frac{\bz}{\norm{\bz}_1}$.
\end{lemma}
\begin{proof} 
  By \eqref{eq:k23d1}, we have $\inner{\nabla h(\bx)-\nabla h(\bx_0)}{ \bz-\bx}= 0$,
  which we can rewrite as
  \begin{align}\label{eq:bound1_1}
   \inner{\log \bx-\log \bx_0}{ \bz-\bx}=\inner{\log(\bx)+\eta}{ \bz-\bx}
    = 0.
  \end{align}
We start by examining the term $\inner{\log(\bx)+\eta}{ \bx}$. Simple manipulations show that
  \begin{align*}
     \inner{\log(\bx)+\eta}{ \bx} &=\norm{\bx}_1\inner{\log(\bx)+\eta}{ \tilde\bx}= \norm{\bx}_1\left(\inner{\log \tilde\bx}{ \tilde\bx}+\log\norm{\bx}_1+\eta\right).
   \end{align*}
   Now by \eqref{eq:bound1_1} we have 
   \begin{align*}
    \norm{\bx}_1-\norm{\bz}_1&= \frac{\inner{\log \bx+\eta}{ \bz}}{\log\norm{\bx}_1+\eta+\inner{\log \tilde\bx}{ \tilde\bx}}-\norm{\bz}_1\\
    &=\norm{\bz}_1\frac{\inner{\log \bx+\eta}{ \tilde\bz}-\log\norm{\bx}_1-\eta-\inner{\log \tilde\bx}{ \tilde\bx}}{\log\norm{\bx}_1+\eta+\inner{\log \tilde\bx}{ \tilde\bx}}\\
    &=\norm{\bz}_1\frac{\inner{\log \tilde\bx}{ \tilde\bz}-\inner{\log \tilde\bx}{ \tilde\bx}}{\log\norm{\bx}_1+\eta+\inner{\log \tilde\bx}{ \tilde\bx}}.
   \end{align*}
   Note that we used here that $\log\norm{\bx}_1+\eta+\inner{\log \tilde\bx}{ \tilde\bx}\neq 0$ which follows from $\norm{\bx}_1>\norm{\bx_0}_1$.
\end{proof}

We now use \Cref{lem:imprbound1} to prove \Cref{pro:imprbound1}. 

 \begin{proof}
 From \Cref{lem:imprbound1} using that $-\inner{\tilde\bx}{\log\tilde\bx}\leq \log(n)$ and $\inner{\log \tilde\bx}{\tilde\bz}\leq \log{\norm{\tilde\bx}_\infty}$
 \begin{align*}
  \norm{\bz}_1\frac{\inner{\log \tilde\bx}{ \tilde\bz}-\inner{\tilde\bx}{\log\tilde\bx}}{\eta+\log\norm{\bx}_1+\inner{\tilde\bx}{\log\tilde\bx}}\leq \norm{\bz}_1\frac{\log{\norm{\tilde\bx}_\infty}-\inner{\tilde\bx}{\log\tilde\bx}}{\eta+\log\frac{\norm{\bx}_1}{n}}.
 \end{align*} 
 
 Thus it remains to maximize the numerator $\log\|\tilde\bx\|_\infty + -\inner{\tilde\bx}{\log\tilde\bx}$ over all probability vectors $\tilde\bx\in\R^n_+$. Fix $t\in[1/n,1]$ and consider all probability vectors with $\|\tilde\bx\|_\infty=t$. 
Among these, $-\inner{\tilde\bx}{\log\tilde\bx}$ is maximized by $\tilde\bx(t)=\Bigl(t,\frac{1-t}{n-1},\ldots,\frac{1-t}{n-1}\Bigr)$ and equals
$-\inner{\tilde\bx(t)}{\log\tilde\bx(t)}=-t\log t-(1-t)\log\Bigl(\frac{1-t}{n-1}\Bigr)$. 

Therefore the maximization reduces to the one-dimensional problem of
\begin{align}
\max_{t\in[1/n,1]}\;
\Bigl(\log t -\inner{\tilde\bx(t)}{\log\tilde\bx(t)}\Bigr)
&=
\max_{t\in[1/n,1]}\;
\Bigl(\log t -t\log t-(1-t)\log\frac{1-t}{n-1}\Bigr)\nonumber\\
&=
\max_{t\in[1/n,1]}\;
(1-t)\log\Bigl(\frac{t(n-1)}{1-t}\Bigr).\label{eq:proof_cor_num}
\end{align}
Differentiation shows that \eqref{eq:proof_cor_num} is maximized for $t^*=\frac{1}{1+W_0(\frac{n-1}{e})}$. The first bound now follows since $(1-t^*)\log\frac{1-t^*}{t^*(n-1)}=W_0(\frac{n-1}{e})$.

To get the tightness result we observe that the worst case for the numerator can be observed. Fix $\bx^* := \tilde\bx(t^*)$ and set $\bz=\lambda(1,0,\ldots)$, where $\lambda=\frac{\eta+\inner{\bx^*}{\log\bx^*}}{\eta+\log t^*}<1.$ Naturally, we can find a linear system such that $\bx^*$, and $\bz$ are the only two vertices in the solution space. Clearly $\bz$ is $\ell_1$-minimal and by \Cref{prop:implicitbias} we have $\bx^*=\argmin_{\bx\in S_+}{D_h(\bx,\bx_0)}.$ 
Substituting $\bx^*$ and $\bz$ into \Cref{lem:imprbound1} we get 
\begin{align*}
 \norm{\bx}_1-\norm{\bz}_1= \norm{\bz}_1\frac{W_0(\frac{n-1}{e})}{\eta+\log\norm{\bx}_1+\log(t^*)-W_0(\frac{n-1}{e})}.
\end{align*}
It remains to show $\log(t^*)-W\leq-\log n+1$, where $W:=W_0\left(\frac{n-1}{e}\right)$.
Since $\log t^*=-\log(1+W)$, this is equivalent to
\begin{align*}
\log(1+W)+W\geq\log(n)-1
\quad\iff\quad
(1+W)e^W\geq\frac{n}{e}.
\end{align*}
By definition of the Lambert function $We^W=\frac{n-1}{e}$ and we have
$(1+W)e^W
=\frac{1+W}{W}\,We^W
=\left(1+\frac1W\right)\frac{n-1}{e}.$ Hence $(1+W)e^W\geq\frac{n}{e}$ is equivalent to $(1+\frac1W)(n-1)\geq n$, i.e. $W\leq n-1$ which holds for all $n\geq1$, proving the claim.
\end{proof}

\section{Proof of linear convergence rates}\label{app:linear_rate}
We first prove three lemmas before stating the main proposition of this section. The main observation used is that on sets strictly separated from the boundary of the nonnegative orthant the Bregman divergence $D_h(\bx,\by)$ can be upper bounded by $\norm{\bx-\by}^2/y_{\min}$. This follows from the descent lemma and the fact that entropy is $L$-smooth on the set $\{\bx\in\R^n_+ : x_{\min}\geq1/L\}$.
\begin{lemma}  \label{lem:xmin_Dhs_x_geq_norm^2}
  Let $\bx\in\R^n_+$ and $\by \in \R^n_{++}$. Then, $ y_{\min}D_h(\bx,\by)\leq \norm{\bx-\by}^2.$
\end{lemma}
\begin{proof}
 Recall that $D_h(\bx,\by)=\inner{\bx}{\log\frac{\bx}{\by}}+\inner{\one}{\by-\bx}$. Thus, it is sufficient to show that $uv\log\frac{u}{v}-uv+v^2\leq (u-v)^2$ for any $u,v\in \R_+$. This holds because it is equivalent to
  $uv\log \frac{u}{v}\leq u^2-uv,$
 which is obvious due to  $ \log t\leq t-1$.
\end{proof}
The following lemmas establish bounds that let us estimate $ y_{\min}$ given the terms of $\bx$ and $D_h(\bx,\by)$. In \Cref{lem:invofDh} we derive the inverse of the Bregman divergence in one dimension and use this to construct a bound on $ y_{\min}$ in higher dimensions in \Cref{lem:lowerboundonxmin}. 

\begin{lemma}[Inverse of Bregman divergence in one dimension]\label{lem:invofDh}
 Let $x\in \R_{++}, y\in \R_+$ and $D_h(x,y)<\infty$. Then $y$ can be expressed in terms of $x$ and the Bregman divergence as follows:
\begin{align*}
y=
\begin{cases}
 -xW_0\left(-\exp \left(-1-\frac{D_h(x,y)}{x}\right)\right), &\text{ if } y\leq x\\
 -xW_{-1}\left(-\exp \left(-1-\frac{D_h(x,y)}{x}\right)\right), &\text{ if } y>x,
\end{cases}
\end{align*}
where $W_0$ and $W_{-1}$ denote the two branches of the Lambert $W$ function.
\end{lemma}
\begin{proof}
We start with the definition of the Bregman divergence
$D_h(x,y)= x\log\frac{x}{y}-x+y$.
Rearranging the terms and dividing by $x$ yields
\begin{align*}
 1+\frac{D_h(x,y)}{x}= \log\frac{x}{y}+\frac{y}{x}.
\end{align*}
Now we apply the function $t\to-\exp(-t)$ to both sides of the equality, resulting in
\begin{align*}
 -\exp \left(-1-\frac{D_h(x,y)}{x}\right)=-\frac{y}{x}\exp\left(-\frac{y}{x}\right).
\end{align*}
Now we can use the Lambert $W$ function, the inverse of the function $t\to t\exp(t)$, to recover $-\tfrac{y}{x}$. Thus,   
\begin{align*}
y=
\begin{cases}
 -xW_0\left(-\exp \left(-1-\frac{D_h(x,y)}{x}\right)\right), &\text{ if } y\leq x\\
 -xW_{-1}\left(-\exp \left(-1-\frac{D_h(x,y)}{x}\right)\right), &\text{ if } y>x.\qedhere
\end{cases}\end{align*}\end{proof}

\begin{lemma}[Lower bound on $ y_{\min}$ given $D_h(\bx,\by)$ and $ x_{\min}$]\label{lem:lowerboundonxmin}
Let $\bx\in \R_{++}^n, \by \in \R_{+}^n$ and $D_h(\bx,\by)<\infty$. Then,
\begin{align*}
 y_{\min} \ge - x_{\min}\,W_0\Bigl(-\exp\Bigl(-1-\frac{D_h(\bx,\by)}{ x_{\min}}\Bigr)\Bigr).
\end{align*}
\end{lemma}

\begin{proof}
For $n=1$, the claim coincides with the equality in \Cref{lem:invofDh}. For $n\geq2$, note that for each coordinate $i$ we have
\begin{align*}
y_i \ge -x_i\,W_0\Bigl(-\exp\Bigl(-1-\frac{D_h(x_i,y_i)}{x_i}\Bigr)\Bigr).
\end{align*}
Since $t\to -W_0\Bigl(-\exp\Bigl(-1-t\Bigr)\Bigr)\geq0$ is decreasing, $x_i\ge  x_{\min}$ and $D_h(x_i,y_i)\le D_h(\bx,\by)$ for every $i$, it follows that
\begin{align*}
y_i \ge - x_{\min}\,W_0\Bigl(-\exp\Bigl(-1-\frac{D_h(\bx,\by)}{ x_{\min}}\Bigr)\Bigr).
\end{align*}
Taking the minimum over all coordinates completes the proof.
\end{proof}

\begin{proof}[Proof of \Cref{thm:linear_rate}]
Let $\bz=\argmin_{\bx\in S_{+}}D_h(\bx, \bx_{0})$. From \Cref{thm:convergence} we know that 
\begin{align*}
 D_h(\bz,\bx_{k+1})-D_h(\bz,\bx_{k})&\leq-\alpha_k f(\bx_k)
\end{align*}
and we can lower bound the stepsize by
\begin{align*}
 \alpha_k\geq \left(4(D_h(\bz,\bx_0) + \norm{\bz}_1){\max_{j\leq n}\norm{A_{:j}}_2^2}\right)^{-1}.
\end{align*}
Since the stepsize is strictly separated from $0$ it only remains to show that $f(\bx_k)\geq c D_h(\bz,\bx_k)$ for some $c>0$. 
First, note that
\begin{align*}
 f(\bx_k)=\frac{1}{2}\norm{A\bx_k-\bb}^2\geq\frac{1}{2}\norm{\bx_k-\bs}^2\lambda_{\min}^{+},
\end{align*} where $\bs\in \R^n$ is the solution to $A\bs=\bb$ that satisfies that $\bx_k-\bs\perp\text{Ker}(A)$ and $\lambda_{\min}^{+}$ is the smallest positive eigenvalue of $A^\top A$.
 Furthermore, since by \Cref{lem:xmin_Dhs_x_geq_norm^2} for any $\bx\in\R^n_{++}$ we have $D_h(\bs,\bx) x_{\min}\leq \norm{\bx-\bs}^2$ we get
\begin{align*}
 \norm{\bx_k-\bs}^2&\geq {D_h(\bs,\bx_k)}(\bx_k)_{\min}\geq \inf_{\bs:A\bs=\bb}D_h(\bs,\bx_k)(\bx_k)_{\min}=D_h(\bz,\bx_k)(\bx_k)_{\min}.
\end{align*} Here in the last equality we use \Cref{prop:implicitbias}, noting that the statement holds not only for $\bx_k$ with $k=0$ but for all $k\geq0$. Furthermore, by \Cref{lem:lowerboundonxmin}, 
\begin{align*} 
 (\bx_k)_{\min}\geq - z_{\min}W_0\left(-\exp\left(-1-\tfrac{D_h(\bz,\bx_k)}{ z_{\min}}\right)\right)
\end{align*}
Therefore, we have 
\begin{align*}
 D_h(\bz,\bx_{k+1})\leq D_h(\bz,\bx_{k})\left(1-\frac{ -\lambda_{\min}^{+} z_{\min}W_0\left(-\exp\left(-1-\tfrac{D_h(\bz,\bx_k)}{ z_{\min}}\right)\right)}{8\max_{j\leq n}\norm{A_{:j}}_2^2\left(\norm{\bz}_1+D_h(\bz,\bx_k)\right)}\right). 
\end{align*}
Since $D_h(\bz,\bx_{k})$ is decreasing this yields a global linear rate, although depending on the initialization this rate can be very slow.

Furthermore, when $\bx_k$ converges to $\bz$ we get a local linear rate of 
\begin{align*}
 1-\frac{\lambda_{\min}^{+}  z_{\min}}{8\max_{j\leq n}\norm{A_{:j}}_2^2\norm{\bz}_1}.\quad\quad\qedhere 
\end{align*}\end{proof}
\section{Proof of \texorpdfstring{\Cref{thm:HDplusconvergence}}{Theorem \ref{thm:HDplusconvergence}}, convergence of Hadamard descent+}\label{app:Hadamard+}
\begin{proof}
The structure of this proof is similar to that of \Cref{thm:convergence}.
We show that using the stepsize as in \eqref{eq:bigHadupdate} allows us to bound the descent of the Bregman divergence as
\begin{align}\label{eq:2f3mp}
 D_h(\bz,\bx_{k+1})-D_h(\bz,\bx_{k})\leq -\alpha_k f(\bx_k).
\end{align}
Then sublinear convergence of function values and convergence of iterates follow by the same arguments as in steps (b) and (c) in the proof of \Cref{thm:convergence}.

By the definition of the Bregman divergence and the update \eqref{eq:bigHadupdate} of $\bx_k$,  we have
\begin{align}
 D_h(\bz,\bx_{k+1})-D_h(\bz,\bx_{k})&= \inner{\bz}{\log\lrb{{\bx_k}/{\bx_{k+1}}}}+ \norm{\bx_{k+1}}_1-\norm{\bx_k}_1\nonumber\\
 &=\inner{\bz}{-\log\left(1-\alpha_k\nabla f(\bx_k)+\alpha_k^2 \nabla f(\bx_k)^2\right)}+\norm{\bx_{k+1}}_1-\norm{\bx_k}_1.\label{eq:hadlog}
\end{align}
By \Cref{lem:phi_positve_increasing} we have that $-\log(1-t+t^2)\leq t$ for $\abs{t}\leq 1.79$. Since by the definition, $\alpha_k\norm{\nabla f(\bx_k)}_\infty \leq1.79$, we can apply this bound and get
\begin{align*}
 \eqref{eq:hadlog}&\leq\inner{\bz}{\alpha_k\nabla f(\bx_k)}+\norm{\bx_{k+1}}_1-\norm{\bx_k}_1\\
 &=\inner{\bz}{\alpha_k\nabla f(\bx_k)}+\inner{\bx_k}{-\alpha_k\nabla f(\bx_k)+\alpha_k^2 \nabla f(\bx_k)^2}\\
 &=-\alpha_k\inner{\bx_k-\bz}{\nabla f(\bx_k)}+\alpha_k^2\inner{\bx_k}{\nabla f(\bx_k)^2}.
\end{align*}
For $f(\bx)=\frac{1}{2}\norm{A\bx-\bb}^2$ and  $\bz\in S_+$ this simplifies to 
\begin{align}
 D_h(\bz,\bx_{k+1})-D_h(\bz,\bx_{k})&\leq -2\alpha_kf(\bx_k)+\alpha_k^2\norm{\nabla f(\bx_k)}^2_{\bx_k}\nonumber\\
 &=-\alpha_k f(\bx_k)+\alpha_k(\alpha_k\norm{\nabla f(\bx_k)}^2_{\bx_k}-f(\bx_k))\label{eq3:line2}\\ 
 &\leq-\alpha_k f(\bx_k).\nonumber
\end{align}
Here in \eqref{eq3:line2} we used that
$\alpha_k\norm{\nabla f(\bx_k)}^2_{\bx_k}\leq f(\bx_k)$. This proves
inequality~\eqref{eq:2f3mp}.
\end{proof}

\section{Proof of lower-boundedness of the stepsize in the convex (locally) \texorpdfstring{$L$}{L}-smooth case}\label{sec:Proof_of_lower_boundedness_L_smooth_case}
The proof follows similarly to the quadratic case, with the main necessary observation being that $L$-smoothness implies that
\begin{align*} f(\bx_k)-f^* \geq \frac{1}{2L}\norm{\nabla f(\bx_k)}_2^2. \end{align*}
\begin{proof}
 We want to show that
\begin{align*} \alpha_k=\min{\frac{f(\bx_k)-f^*}{2\norm{\nabla f(\bx_k)}_{\bx_k}^2},\frac{1.79}{\normi{\nabla f(\bx_k)}}} \end{align*}
is bounded away from zero. We will first consider the first term in the minimum.
By Hölder's inequality, we have
\begin{align*} \frac{f(\bx_k)-f^*}{2\norm{\nabla f(\bx_k)}_{\bx_k}^2} &\geq \frac{f(\bx_k)-f^*}{2\normi{\bx_k}\norm{\nabla f(\bx_k)^2}_1}= \frac{1}{2\normi{\bx_k}}\frac{f(\bx_k)-f^*}{\norm{\nabla f(\bx_k)}_2^2}\geq \frac{1}{4L\normi{\bx_k}}, \end{align*}
where the last inequality follows from the fact that $L$-smoothness implies that \begin{align*} f(\bx_k)-f^*\geq \frac{1}{2L}\norm{\nabla f(\bx_k)}_2^2. \end{align*}
For the second term note that by $L$-smoothness \begin{align*} \frac{1.79}{\norm{\nabla f(\bx_k)}_\infty}&\geq\frac{1}{\norm{\nabla f(\bx_k)}_2}\geq\frac{1}{L\norm{\bx_k-\bz}_2}. \end{align*} Combining the above inequalities, we can bound $ \alpha_k $ by \begin{align*} \alpha_k\geq \frac{1}{L}(\max{\norm{\bx_k-\bz}_2,4\normi{\bx_k}})^{-1}\geq \frac{1}{L}(\max{\norm{\bx_k-\bz}_1,4\norm{\bx_k}_1})^{-1}. \end{align*}
By \Cref{lem:genPinskerbound}, we have that \begin{align*} \norm{\bx_k}_1\leq 2D_h(\bz,\bx_k) + 2\norm{\bz}_1 \end{align*} and by the triangle inequality, \begin{align*} \norm{\bz-\bx_k}_1\leq \norm{\bz}_1+\norm{\bx_k}_1\leq 2D_h(\bz,\bx_k) + 3\norm{\bz}_1. \end{align*}
Since the sequence $D_h(\bz,\bx_k)$ is nonincreasing, we deduce that
\begin{align*} \max{4\norm{\bx_k}_1, \norm{\bx_k-\bz}_1}\leq 8(D_h(\bz,\bx_k) + \norm{\bz}_1)\leq 8(D_h(\bz,\bx_0) + \norm{\bz}_1). \end{align*}

Thus, \begin{align*} \alpha_k\geq \frac{1}{8L(D_h(\bz,\bx_0) + \norm{\bz}_1)}, \end{align*} which proves lower-boundedness. Now, using this estimate in \eqref{eq:32rf3q} (with an additional factor of $\frac12$ because we are in the convex case), we obtain
\begin{align*} \min_{0\leq i\leq k}f(\bx_i)-f^*\leq \frac{16LD_h(\bz,\bx_{0})(D_h(\bz,\bx_0) +\norm{\bz}_1)}{k+1}. \end{align*}
Note that since our iterates are bounded, local smoothness would also have been sufficient to establish sublinear convergence.
\end{proof}

\section{Proof of \texorpdfstring{\Cref{prop:diverge}}{Proposition \ref{prop:diverge}}, instability of mirror descent with constant stepsizes independent of \texorpdfstring{$\boldsymbol{b}$}{b}}\label{app:unstable}
  \begin{definition}
    Consider a dynamical system $\bx_{t+1} = F(\bx_t)$ for $F\colon \R^n \to
    \R^n$ and let $\bx^*$ be its fixed point, $F(\bx^*)=\bx^*$.  We call $\bx^*$ \emph{unstable}, if there exists $ \varepsilon > 0 $ such that for any $\delta > 0 $, there exists $ \bx_0 $ with $ 0 < \norm{\bx_0 - \bx^*} < \delta$ for which 
\begin{align*}
\limsup_{t \to \infty} \norm{\bx_t - \bx^*} \geq \varepsilon. 
\end{align*}
\end{definition}
\begin{proof}
Consider the mirror descent update function
\begin{align*}
F_{\bb}(\bx) = \bx\circ \exp(-\alpha \nabla f_{\bb}(\bx)),\quad \text{with} \quad f_{\bb}(\bx)=\frac{1}{2}\norm{A\bx-\bb}^2.
\end{align*} The fixed points of this dynamical system are the points where the gradient vanishes (i.e.~ the solutions of the linear system).
Calculating the Jacobian at a solution $\bx$ yields
\begin{align*}
JF_{\bb}(\bx) = I - \alpha\,\diag(\bx)A^\top A.
\end{align*}
Thus, if 
\begin{align*}
\alpha > \frac{2}{\lambda_{\max}(\diag(\bx) A^\top A)},
\end{align*}
where $\lambda_{\max}(\diag(\bx) A^\top A)$ is the largest eigenvalue of $\diag(\bx )A^\top A$, then $\bx $ is an unstable fixed point.

Now, let $\bx^*\in\R^n_+$ be a solution of $A\bx=\bb$ that minimizes $\lambda_{\max}(\diag(\bx) A^\top A)$. Notice that if we scale $\bx^*$ by any $t>0$, defining
\begin{align*}
\tilde{\bx}^* = t\, \bx^* \quad \text{and} \quad \tilde{\bb} = A\tilde{\bx}^* = t\,\bb,
\end{align*}
then 
\begin{align*}
\lambda_{\max}(\diag(\tilde{\bx}^*)A^\top A) = t\,\lambda_{\max}(\diag(\bx^*)A^\top A)
\end{align*}
and $\tilde{\bx}^*$ minimizes the top eigenvalue of $\diag(\bx )A^\top A$ among all solutions of  $A\bx=\tilde{\bb}$.

For a given $\alpha$, set  
\begin{align*}
t = \frac{3}{\alpha\,\lambda_{\max}(\diag(\bx^*) A^\top A)}.
\end{align*}
Then 
\begin{align*}
  \frac{2}{\lambda_{\max}(\diag(\tilde{\bx}^*) A^\top A)}=\frac{2}{t\lambda_{\max}(\diag(\bx^*) A^\top A)}= \frac{2\alpha}{3}\leq \alpha 
  ,
\end{align*}
and since $\tilde\bx$ minimizes $\lambda_{\max}(\diag(\tilde{\bx}^*) A^\top A)$ we get that all solutions are unstable fixed points of the dynamical system
$\bx_{k+1}=F_{\tilde\bb}(\bx_{k})$.
\end{proof}

\section{Numerical experiments}\label{app:numerics}
In this section, we present numerical experiments to support our theoretical conclusions. In particular we will compare entropic mirror descent and Hadamard descent with the Polyak stepsize to mirror descent with the optimal constant stepsize and with backtracking. Additionally, we investigate how initialization near the origin influences the convergence speed of entropic mirror descent with the Polyak stepsize. All code used to generate the numerical experiments is publicly available on \href{https://github.com/AlexanderPosch/Numerics-for-EMD-for-Lin.-Syst.-Polyak-Stepsize-and-Implicit-Bias}{GitHub}.}.
\subsection{Experiment 1: convergence speed comparison}

In this experiment we compare the convergence speed of five algorithms:
\begin{itemize}
    \item[](MD-const.) Mirror descent with the optimal constant stepsize.
    \item[](MD+backtracking) Mirror descent with backtracking and local relative smoothness.
    \item[](MD-Polyak) Mirror descent with Polyak’s stepsize \eqref{eq:alg}.
    \item[](HD-Polyak) Hadamard descent with Polyak’s stepsize \eqref{eq:f322f}.
    \item[](HD+Polyak) Hadamard descent + with Polyak’s stepsize
    \eqref{eq:bigHadupdate}.
\end{itemize}

We generate a $300\times 500$ matrix $A$ whose nonzero singular values are i.i.d. standard half-normal distributed. A target vector $\bz\in [0,1]^{500}$, with $\norm{\bz}_0=30$, is sampled uniformly, and we set $\bb = A\bz$. Each algorithm is initialized with $\bx_0 = 10^{-4}\mathbf{1}$ and run for 25,000 iterations. To estimate the limit, we run an additional 25,000 iterations and record the final vectors.

The left panel of Figure~\ref{fig:experiment10} plots the cumulative minimum function value, i.e., 
\begin{align*}
f(\bar{\bx}_k) - f^*, \quad \text{with } \bar{\bx}_k = \argmin_{s\le k} f(\bx_s), 
\end{align*}
while the right panel shows the Bregman divergence to the estimated limit. The results indicate that mirror descent with Polyak’s stepsize leads to faster convergence than both the optimal constant stepsize and backtracking. Additionally, the Hadamard descent variants with Polyak's stepsize seem to perform similarly to mirror descent.  Note that the sequences of function values are not monotone for the algorithms with the Polyak stepsize.

\begin{figure}[ht!]
    \centering
    \includegraphics[width=1\textwidth,alt={Two side-by-side line plots comparing convergence speed of multiple first-order methods for nonnegative least squares. The left panel shows the cumulative minimum function value, and the right panel shows the Bregman divergence to the estimated limit. In both plots, mirror descent with Polyak’s stepsize converges faster than mirror descent with the optimal constant stepsize and with backtracking. The Hadamard descent variants with Polyak's stepsize perform similarly to mirror descent with Polyak's stepsize.}]{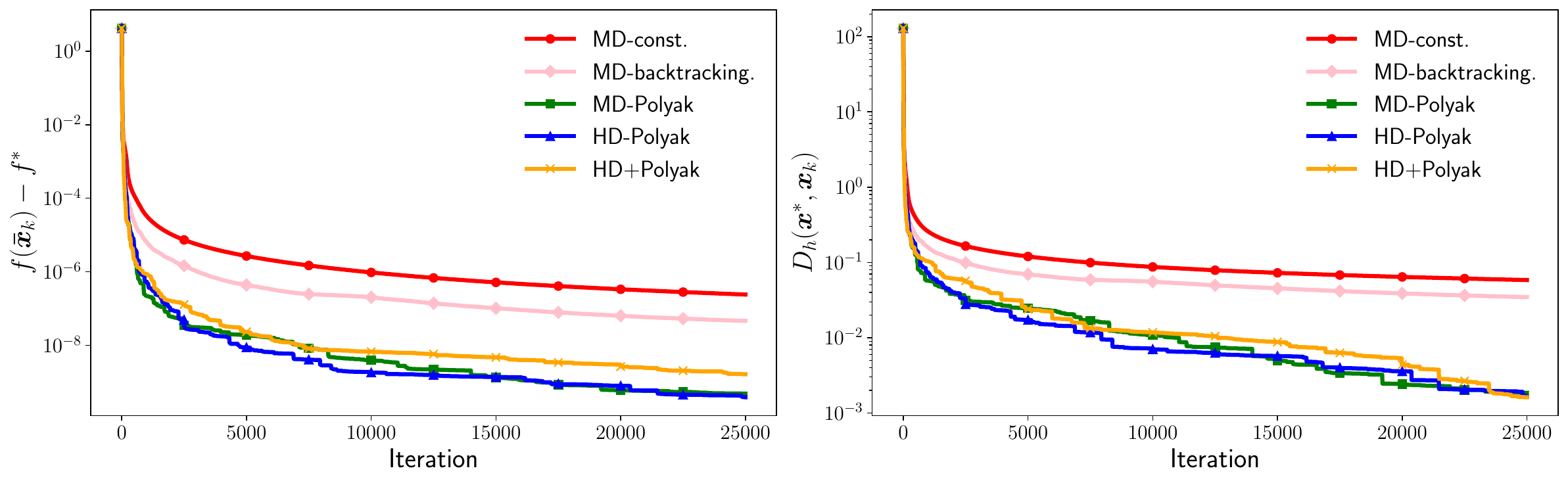}
    \caption{\footnotesize Left: Cumulative minimum function value $f(\bar{\bx}_k)-f^*$ for MD-const., MD-Polyak, HD-Polyak, HD+Polyak and MD-backtracking. Right: Bregman divergence to the estimated limit.\label{fig:experiment10}}
\end{figure}

\subsection{Experiment 2: effect of initialization size on convergence speed}

We study how initialization close to $0$ affects convergence speed of mirror descent with Polyak's stepsize for solving both sparse and nonsparse linear systems.

We generate a $300\times 500$ matrix $A$ with standard normally distributed eigenvalues and sample two target vectors $\bz_1,\bz_2 \in [0,1]^{500}$, where $\bz_1$ is modified to satisfy $\norm{\bz_1}_0=30$, while $\bz_2$ is dense. We set $\bb_1=A\bz_1$ and $\bb_2=A\bz_2$. 

For each starting vector 
\begin{align*}
\bx_0 = 10^{-2}\mathbf{1},\, 10^{-4}\mathbf{1},\, 10^{-8}\mathbf{1},\, 10^{-16}\mathbf{1},\, 10^{-32}\mathbf{1}, 
\end{align*}
we run mirror descent with Polyak’s stepsize for 25,000 iterations. Figure~\ref{fig:experiment2} shows the cumulative minimum function value (log-log scale). In the sparse case (left), smaller initializations slow early convergencebut, somewhat surprisingly, eventually outperform larger initializations. In the dense case (right), initialization close to $0$ consistently leads to slower convergence. 

\begin{figure}[htb!] 
    \centering
    \includegraphics[width=1\textwidth,alt={Two log-log line plots showing cumulative minimum function values versus iteration for different magnitudes of initialization. In the sparse case (left), very small initializations slow early convergence but eventually outperform larger initializations. In the dense case (right), initializations close to zero consistently lead to slower convergence.}]{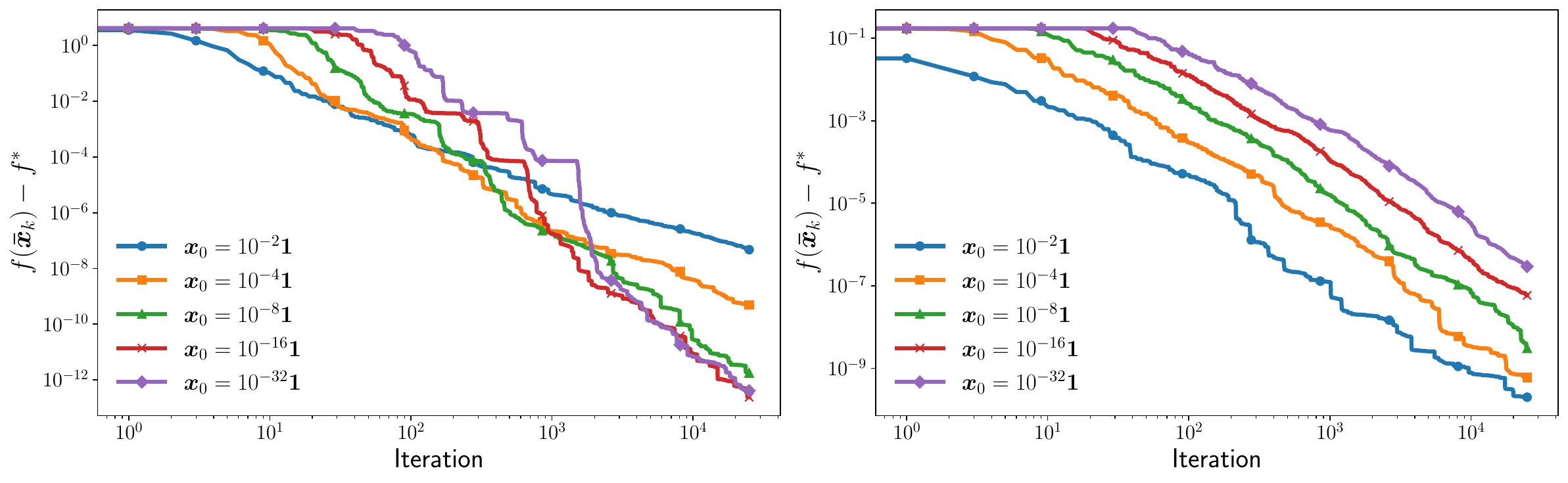}
    \caption{\footnotesize  Cumulative minimum function value $f(\bar{\bx}_k)-f^*$ for mirror descent with Polyak’s stepsize under various initialization. Left: Sparse solution ($\norm{\bz_1}_0=30$). Right: Dense solution.\label{fig:experiment2}}
\end{figure}
\end{document}